\documentclass[letterpaper, 12pt]{article}
\usepackage{amsmath, amsthm, amssymb}
\usepackage{thmtools}
\usepackage{thm-restate}
\usepackage[numbers,sort]{natbib}
\usepackage{bm}
\usepackage{graphicx}
\usepackage{fullpage}
\usepackage[colorlinks,linkcolor=blue,citecolor=blue,urlcolor=magenta,linktocpage,plainpages=false]{hyperref}\usepackage{xcolor}
\usepackage{xspace}
\usepackage{url}
\usepackage[noend]{algpseudocode}
\usepackage{algorithm}
\usepackage{libertine}
\usepackage{libertinust1math}
\usepackage{caption}
\usepackage{setspace}
\usepackage[left=1in, right=1in, top=1in, bottom=1in]{geometry}
\usepackage{authblk}

\newtheorem{theorem}{Theorem}

\newtheorem{lemma}[theorem]{Lemma}
\newtheorem{corollary}[theorem]{Corollary}
\newtheorem{prop}[theorem]{Proposition}
\newtheorem{remark}{Remark}
\newtheorem{fact}{Fact}

\newcommand{\inftynorm}[1]{\left\lVert#1\right\rVert_\infty}
\newcommand{\zeronorm}[1]{\left\lVert#1\right\rVert_0}
\newcommand{\rounddown}[1]{\left\lfloor#1\right\rfloor}
\newcommand{\roundup}[1]{\left\lceil#1\right\rceil}
\newcommand{\Z}{\mathbb{Z}}
\newcommand{\R}{\mathbb{R}}

\DeclareMathOperator{\intcone}{int\_cone}
\DeclareMathOperator{\cone}{cone}

\title{On the Smallest Support Size of Integer Solutions to Linear Equations}

\author{Yatharth Dubey\thanks{Amazon. Email: ydubey0711@gmail.com}}
\author{Siyue Liu\thanks{Tepper School of Business, Carnegie Mellon University. Email: siyueliu@andrew.cmu.edu}}

\affil{}
\date{}

\begin{document}
\maketitle

\begin{abstract}
In this note, we study the size of the support of integer solutions to linear equations $Ax=b, ~x\in\Z^n$ where $A\in\Z^{m\times n}, b\in\Z^n$. 
We give an upper bound on the smallest support size as a function of $A$, taken as a worst case over all $b$ such that the above system has a solution. This bound is asymptotically tight, and in fact matches the bound given in Aliev, Averkov, De Loera and Oertel \cite{aliev2022sparse}, while the proof presented here is simpler, relying only on linear algebra. 
\end{abstract}

\section{Introduction}
\label{sec:intro}
The main goal of this work is to establish upper and lower bounds on the smallest support size of integer solutions to a system of linear equations. Let $A \in \Z^{m \times n}$ be a matrix and denote by $\mathcal{L}(A) :=\{Ax\mid x\in\Z^n\}$ the lattice generated by the columns of $A$. For any $b\in \mathcal{L}(A)$, we want to find an integer solution to $Ax=b$ with smallest support size, i.e.
\begin{equation}\label{eq:main_prob}
    \min~\{\|x\|_0 \mid Ax = b, x \in \Z^n\},
\end{equation}
where $\|x\|_0 := |\{j : x_j \not = 0\}|$. We will assume without loss of generality that $A$ has full row rank, i.e. $\text{rank}(A) = m$, because the system is trivially satisfied if and only if the equations corresponding to independent rows of $A$ are satisfied. 
In this note, we study the largest value of \eqref{eq:main_prob} taken over all points $b$ in the lattice $\mathcal{L}(A)$. To be precise, for matrix $A\in\Z^{m\times n}$, define 
\begin{equation}
\label{eq:def:min:support}
    f(A):=\max\limits_{b\in \mathcal{L}(A)}\min~\{\|x\|_0 \mid Ax = b, x \in \Z^n\}.
\end{equation}
This quantity has been the focus of several previous studies. For example, in \cite{aliev2022sparse}, $f(A)$ is called the \textit{integer linear rank} of $A$, denoted by $ILR(A)$. According to \cite{aliev2022sparse}, determining $f(A)$ is NP-hard. Their proof actually implies that given both $A,b$ as inputs, determining the value of \eqref{eq:main_prob} is NP-hard. Nontrivial bounds on problem \eqref{eq:def:min:support} have been leveraged in many areas of discrete optimization; see, for example, references in Section 1 of \cite{aliev2022sparse} and in the introduction of \cite{aliev2018support}. 

As a result, there have been several studies focused on achieving increasingly strong upper bounds on the quantity (\ref{eq:def:min:support}). Currently, the best available upper bound of $f(A)$ is achieved in \cite{aliev2022sparse}; let $f_{AADLO}(A)$ denote this bound. Let $A$ be an integer matrix with $m$ rows and at least $m$ columns, with largest absolute entry denoted by $\|A\|_\infty$. Their bound is parameterized by the following quantities. 
\begin{equation*}
\begin{aligned}
    \Gamma(A):=&\max\limits ~\{|\det(B)|\mid \text{$B$ is an $m\times m$ submatrix of $A$}\};\\
    \gcd(A):=&\gcd ~\{|\det(B)|\mid \text{$B$ is an $m\times m$ submatrix of $A$}\}.
\end{aligned}
\end{equation*}
The main contribution of this work is the following theorem, which essentially matches the state-of-the-art bound $f_{AADLO}(A)$.
\begin{theorem}
\label{thm:main}
    Let $A$ be an integer matrix with $m$ rows and full row rank. Let $p_i$ be the $i$-th prime. Then, the following relation between $\Gamma(A)$, $\gcd(A)$ and $f(A)$ holds.
    \begin{equation}
    \label{eq:integer:equation:bound}
        \frac{\Gamma(A)}{\gcd(A)}\geq p_2^m p_3^m\cdots p_{\rounddown{\frac{f(A)}{m}}}^{m}p_{\roundup{\frac{f(A)}{m}}}^{f(A)-m\rounddown{\frac{f(A)}{m}}}.
    \end{equation}
    Furthermore, the bound is asymptotically tight in the sense that for arbitrarily large integers $m$ and $t$, there exist an integer matrix $A$ with $m$ rows, full row rank, $\inftynorm{A}>t$, and a vector $b\in\mathcal{L}(A)$ such that equality holds in \eqref{eq:integer:equation:bound}.
\end{theorem}
We will see how the bound \eqref{eq:integer:equation:bound} is comparable to $f_{AADLO}(A)$; in particular, the proof of Theorem \ref{thm:main} can easily be adapted to obtain the bound $f_{AADLO}(A)$.



We are interested in obtaining bounds on $f(A)$ parameterized by the number of rows and the largest absolute entry of $A$. With this in mind, define 
\begin{equation}
\begin{aligned}
    \label{eq:def:maxmin:support}
    h(m,t):=&\max\limits_{A\in \mathbb{Z}^{m\times n},\inftynorm{A}\leq t, b\in \mathcal{L}(A)}\min~\{\|x\|_0 \mid Ax = b, x \in \Z^n\}\\
    =&\max\limits_{A\in \mathbb{Z}^{m\times n},\inftynorm{A}\leq t}f(A).
\end{aligned}
\end{equation}
The following bound follows as a corollary.
\begin{theorem}
    \label{thm:integer:equation:support:bound}
    \begin{equation}
    \label{eq:log:bound}
        h(m,t)\leq O\left(\frac{m\log(\sqrt{m}~t)}{\log\log(\sqrt{m}~t)}\right). 
    \end{equation}
\end{theorem}


We intend for this paper to stand out for the strength of its result, but also for its simplicity and readability. 

\paragraph{Overview.} In Section \ref{sec:literature}, we discuss related work and how the results of this paper compare with recent developments in bounding the support size of integer solutions to systems of equations/inequalities. In Section \ref{sec:integral_independence} we introduce the notion of integral independence, which will be useful throughout the proofs of Theorems \ref{thm:integer:equation:support:bound} and \ref{thm:main}. In particular, in Section \ref{sec:bound_f}, we achieve a nontrivial lower bound on the largest $m \times m$ subdeterminant of a matrix $A$ with integrally independent columns, given by Theorem \ref{thm:main}. In Section \ref{sec:bound_h}, using the relationship between the largest subdeterminant and largest absolute entry given by the Hadamard's inequality \cite{hadamard1893resolution}, we achieve the upper bound on the number of columns, i.e. $h(m,t)$, given by Theorem \ref{thm:integer:equation:support:bound}. Finally, in Section \ref{sec:tight_example} we exhibit the matrix $A$ referred to in the latter statement of Theorem \ref{thm:main}.

\paragraph{Notation.} Throughout this paper we use $p_i$ to denote the $i$-th prime number. In particular, $p_1 =2$. For any $p,q \in \mathbb{Z}$, we use $p \mid q$ to denote that $p$ divides $q$, i.e. $q/p \in \mathbb{Z}$; similarly, we use $p \nmid q$ to denote that $p$ does not divide $q$; we may use this interchangeably with the language $q$ is (respectively, is not) divisible by $p$. We use the convention that $\gcd(a,0) = |a|$ for any $a \in \Z$, including $0$. We use the notation $[n]:=\{1,2,...,n\}$, and, for positive integers $i\leq j$, $[i:j]:=\{i,i+1,...,j\}$. For any matrix $A\in\R^{m\times n}$, and $I\subseteq [m]$, $J\subseteq [n]$, we denote by $A_{I\times J}$ the submatrix of $A$ consisting of rows indexed by $I$ and columns indexed by $J$. Denote by $\inftynorm{A}$ the infinity norm of $A$. For $x \in \R$, denote by $\rounddown{x}$ the largest integer less than or equal to $x$; $\roundup{x}$ the smallest integer greater than or equal to $x$. We use $\Z^n_{\geq 0} := \{x \in \Z^n \mid x_j \geq 0 \text{ for all } j \in [n] \}$, and $\R^n_{\geq 0}$ analogously.

\section{Related work}
\label{sec:literature}
Problem \eqref{eq:def:min:support} has garnered interest from multiple research disciplines. For example, an interesting problem is to find nonnegative integer solutions to $Ax=b$. Define the \textit{integer conic hull} generated by the columns of $A$ to be the vectors that can be represented as nonnegative integer combinations of the columns of $A$, denoted as $\intcone(A):=\{Ax\mid x\in \Z_{\geq 0}^n\}$.
Let
\begin{equation}
\begin{aligned}
    \label{eq:def:maxmin:support:ineq}
    g(m,t):=\max_{A\in \mathbb{Z}^{m\times n},~\inftynorm{A}\leq t,~ b \in \intcone(A)}\min~\{\|x\|_0 \mid Ax = b, x \in \Z_{\geq 0}^n\}.
\end{aligned}
\end{equation}
Note that $h(m,t)\leq g(m,t)$. Indeed, for any $b\in\mathcal{L}(A)$, there exists $A'$, which can be obtained upon negating some columns of $A$, such that $b\in\intcone(A')$. Take any solution $x'$ to $A'x'=b,~x'\in \Z_{\geq 0}^n$ and upon negating some coordinates of $x'$, we obtain a solution to $Ax=b,~x\in \Z^n$ of the same support as $x'$. Thus an upper bound on $g(m,t)$ can serve as an upper bound for $h(m,t)$. Eisenbrand and Shmonin \cite{eisenbrand2006caratheodory} establish the first upper bound on $g(m,t)$ using the pigeonhole principle, achieving $g(m,t)\leq 2m\log (4mt)$. This bound has been improved by \cite{aliev2018support} to $g(m,t)\leq 2m\log (2\sqrt{m}t)$ using Siegel's lemma \cite{bombieri1983siegel}.
Using $g(m,t)$ to upper bound $h(m,t)$, one can obtain $h(m,t)\leq 2m\log (2\sqrt{m}t)$, which is the best upper bound on the integer linear rank, parameterized by the number of rows and largest absolute entry of the matrix, available in the literature. Note that the bound of Theorem \ref{thm:integer:equation:support:bound} improves these bounds on $h(m,t)$ by a factor of  $O(\log\log(\sqrt{m}t))$. An asymptotic lower bound for $g(m,t)$ is also established in \cite{aliev2018support}: for any $\epsilon>0$, there exist a matrix $A\in\Z^{m\times n}$ with $n/m$ large enough and a vector $b\in\intcone(A)$, such that $\min~\{\|x\|_0 \mid Ax = b, x \in \Z_{\geq 0}^n\}\geq m\log(\inftynorm{A})^{\frac{1}{1+\epsilon}}$.

Another interesting direction is when the columns of matrix $A$ form a \textit{Hilbert basis}, i.e., for any $b$ in $\cone(A)\cap \Z^n$, $b$ is also in $\intcone(A)$, where $\cone(A):=\{Ax\mid x\in\R_{\geq 0}^n\}$. Cook et al. \cite{cook1986integer} show that when $\cone(A)$ is pointed and the columns of $A$ form an integral Hilbert basis, {the smallest support size of solutions to $Ax=b,~x\in\Z_{\geq 0}^n$ is always upper bounded by $2m-1$.} Seb\"{o} \cite{sebo1990hilbert} improves this bound to $2m-2$. Note that the bound is independent of $\inftynorm{A}$.

As far as we know, the most relevant paper in the literature is \cite{aliev2022sparse}, where they establish an upper bound for $f(A)$ and show that it is optimal. For any $z\in\Z_{>0}$, consider the prime factorization $z=q_1^{s_1}\cdots q_k^{s_k}$ such that $q_1,...,q_k$ are pairwise distinct. Aliev et al. \cite{aliev2022sparse} introduce $\Omega_m(z):=\sum_{i=1}^{k}\min\{s_i,m\}$, called \textit{truncated prime $\Omega$-funciton}. Let $\binom{[n]}{m}$ be all the subsets of $[n]$ of cardinality $m$ and for $\tau\in \binom{[n]}{m}$, let $A_{\tau}$ be the $m\times m$ submatrix of $A$ with columns indexed by $\tau$. They show 
\begin{equation}
\label{eq:bound:Omega}
    \begin{aligned}
        f(A)\leq m+\min_{\tau\in\binom{[n]}{m}, \det(A_\tau)\neq 0}\Omega_m\left(\frac{|\det(A_\tau)|}{\gcd(A)}\right).
    \end{aligned}
\end{equation}
They also show this bound is optimal in the sense that neither $m$ can be replaced by any smaller constant nor the function $\Omega_m$ can be replaced by any smaller function. 

The proof of bound \eqref{eq:bound:Omega} relies on the connection between the theory of finite Abelian groups and lattice theory. In particular, they use the \textit{primary decomposition} theorem of finite Abelian groups (see e.g. Chapter 5.2 in \cite{dummit2004abstract}) and group representation of lattices (see e.g. Section 4.4 of \cite{schrijver1998theory}). In contrast, the approach in this paper is significantly simpler, only using linear algebra. The proof is self-contained and does not require knowledge of group theory or lattice theory. Moreover, we will show how the bound \eqref{eq:bound:Omega} can be obtained as a byproduct of our proof of Theorem \ref{thm:main}.

\section{Proof of Theorems \ref{thm:main} and \ref{thm:integer:equation:support:bound}}
\label{sec:proof}

\subsection{Integral independence}
\label{sec:integral_independence}

A set of vectors $\{a_1,...,a_n\}$ is called \emph{integrally dependent} if there is some $k \in [n]$ such that $a_k = \sum_{i \not = k} \mu_i a_i$, where every $\mu_i $ is an integer, and \emph{integrally independent} otherwise. Furthermore, if $A$ is a matrix with integrally independent columns, we call the matrix itself integrally independent. 
The integral independence is closely related to the smallest support of integral solutions to an equation $Ax=b$, which we will elaborate on in the following proposition.

\begin{prop}\label{prop:int_indep}
Let $A$ be a matrix with columns $a_1,...,a_n \in \mathbb{Z}^m$. Then, the following are equivalent:
\begin{enumerate}
    \item The vectors $\{a_1,...,a_n\}$ are integrally independent;
    \item There exists a vector $b \in \mathcal{L}(A)$ such that $\zeronorm{x} = n$ for any integer solution $x$ to $Ax = b$;
    \item There is no integer vector in the null space of $A$ with an entry equal to $1$, i.e. $\{x \in \mathbb{Z}^n : Ax = 0,~ \exists j\in [n] \text{ such that } x_j = 1\} = \emptyset$.
\end{enumerate}
\end{prop}

\begin{proof}
We start by proving $(1) \implies (2)$. Let $b = \sum_{i \in [n]} a_i $. Consider an integral vector $x\in \mathbb{Z}^{n}$ such that $\sum_{i \in [n]} a_i x_i = \sum_{i \in [n]} a_i$. For sake of contradiction, assume some component of $x$ is zero; without loss of generality, assume $x_n = 0$. Then, $\sum_{i \in [n-1]} a_i x_i = \sum_{i \in [n]} a_i$. It follows that $a_n = - \sum_{i \in [n-1]} (1 - x_i) a_i$, where $a_n$ is an integral combination of $\{a_1,...,a_{n-1}\}$, a contradiction.

We now prove $(2) \implies (1)$. Let $b\in\mathcal{L}(A)$ be such that every integral $x$ with $Ax = b$ has $\zeronorm{x} = n$. Further let $x'$ be one such choice of $x$ and write $b = \sum_{i \in [n]} a_i x'_i $. Suppose for sake of contradiction that $\{a_1,...,a_n\}$ is integrally dependent. Then, by definition, there is some $k \in [n]$ and $\mu\in\mathbb{Z}^{n-1}$ such that $a_k = \sum_{i \not = k} \mu_i a_i$. So, we can write
$$ b = \sum_{i \not = k} a_i x'_i + \sum_{i \not = k}x'_k \mu_i a_i = \sum_{i \not = k} (x'_i + \mu_i x'_k)a_i. $$
It follows from $x'$ and $\mu$ are both integral that each $(x'_i + \mu_i x'_k)$ is integral. Then, $b$ can be written as an integral combination of $\{a_1,...,a_n\} \setminus \{a_k\}$, contradicting the fact that $\zeronorm{x}=n$ for every integral $x$ such that $Ax = b$.

We now prove $(1) \implies (3)$. Suppose there is some $x \in \mathbb{Z}^n \setminus \{0\}$ such that $\sum_{i \in n} a_i x_i = 0$; and without loss of generality, assume $x_n = 1$. Then, we can rewrite this as $a_n = -\sum_{i \in [n-1]}\frac{x_i}{x_n} a_i=\sum_{i \in [n-1]}(-x_i) a_i$. Therefore, $\{a_1,...,a_n\}$ is integrally dependent.

We now prove $(3) \implies (1)$. Suppose that $\{a_1,...,a_n\}$ is integrally dependent. Then, by definition, there exists a $k \in [n]$ such that $a_k = \sum_{i \not = k} \mu_i a_i$, where every $\mu_i$ integer. Therefore, vector \newline $\begin{pmatrix}
    -\mu_1 & -\mu_2 & ... & - \mu_{k-1} & 1 & -\mu_{k+1} & ... & -\mu_n
\end{pmatrix}$ is in the null space of $A$.
\end{proof}

\subsection{Upper bound on $f(A)$}
\label{sec:bound_f}


Applying unimodular column operations, we can bring $A$ into Hermite normal form $H=\begin{pmatrix} D & 0 \end{pmatrix}$, where $D\in \mathbb{Z}^{m\times m}$ is a lower triangular matrix. We have the relation $AU=\begin{pmatrix} D & 0 \end{pmatrix}$, where $U\in \mathbb{Z}^{n\times n}$ is a unimodular matrix corresponding to the unimodular operations (see e.g. Section 1.5.2 in~\cite{conforti2014integer}). To proceed, we remind readers of some basic facts about unimodular operations.

{Define the GCD of zeros to be zero, and note that the GCD of several numbers is zero if and only if they are all equal to zero. Recall that for $C\in\Z^{m\times n}$ with $m\le n$, $gcd(C)$ is the GCD of all $m\times m$ subdeterminants of $C$.

\begin{fact}
\label{fact:1}
    Let $C\in \Z^{m\times n}$ with $m \le n$ and $V \in \Z^{n x n}$ be a unimodular matrix. Then, $gcd(CV)=gcd(C)$. 
\end{fact}

\begin{fact}
\label{fact:2}
    $gcd(C)=0$ if and only if the rows of $C$ are linearly dependent.
\end{fact}
}

The first fact follows from the GCD of all $m\times m$ subdeterminants of $C$ being invariant under unimodular column operations (see, e.g., Section 4.4 of \cite{schrijver1998theory}). To see the second fact, we can bring $C$ into its Hermite normal form $\begin{pmatrix} C' & 0 \end{pmatrix}$ by applying unimodular column operations $V$, i.e. $CV=\begin{pmatrix} C' & 0 \end{pmatrix}$. By the first fact, the GCD of all $m\times m$ subdeterminants of $C$ equals that of $CV$, which is $\det(C')$. Thus, the GCD is $0$ if and only if $\det(C')=0$, which means the rows of $C'$ are linearly dependent. Since $V$ is unimodular and thus nonsingular, the rows of $C'$ are linearly dependent if and only if the rows of $C$ are linearly dependent. Furthermore, it may be useful to recall that the Hermite normal form of a matrix is unique. Due to these facts and the properties of unimodular matrices, we have the following lemma.

\begin{lemma}\label{lemma:propertynull}
    Let $A\in\mathbb{Z}^{m\times n}$ be a matrix with integrally independent columns. Let $H = AU$ be the Hermite normal form of $A$, where $U$ is a unimodular matrix. Suppose  $U=\begin{pmatrix} U_1 & U_2 \end{pmatrix}$, where $U_1\in \mathbb{Z}^{n\times m}, U_2\in \mathbb{Z}^{n\times (n-m)}$. Then $U_2$ must satisfy the following properties:
\begin{enumerate}
    \item[(P1).] $U_2$ has at least one nonsingular $(n-m)\times (n-m)$ submatrix.
    \item[(P2).] The GCD of all $(n-m)\times (n-m)$ subdeterminants of $U_2$ is $1$.
    \item[(P3).] For each row $i$ of $U_2$, there exists some prime $q_i \geq 2$ such that $q_i\mid (U_2)_{ij}$ for all $j \in [n-m]$.
\end{enumerate}
\end{lemma}

\begin{proof}
    It follows from $U^{-1}U = I$ that $U_2^\top (U^{-1})^\top=\begin{pmatrix} I_{n-m} & 0 \end{pmatrix}$. Together with the fact that $(U^{-1})^\top$ is unimodular we obtain that the Hermite normal form of $U_2^\top$ is $\begin{pmatrix} I_{n-m} & 0 \end{pmatrix}$. According to Fact \ref{fact:1}, the $gcd(U_2^\top)=\det(I_{n-m})=1$, proving (P2). In particular, it is nonzero and thus by Fact \ref{fact:2}, $U_2$ has rank $n-m$ and therefore has a nonsingular $(n-m)\times (n-m)$ submatrix, proving (P1).
    
    It remains to prove property (P3). Since $A\begin{pmatrix} U_1 & U_2 \end{pmatrix}=\begin{pmatrix} D & 0 \end{pmatrix}$, for $D\in \mathbb{Z}^{m\times m}$ some lower triangular matrix, we notice that columns of $U_2$ are in the null space of $A$. Since the columns of $A$ are integrally independent, by Proposition \ref{prop:int_indep}, any integral $x$ in the null space of $A$ has no entry equal to $1$. Suppose for the sake of contradiction, the entries in row $i$ of $U_2$ have GCD $1$. Then, there exists an integral $\mu\in\mathbb{Z}^{n-m}$ such that $\sum_{j=1}^{n-m}\mu_j(U_2)_{ij}=1$.  Then $x = U_2 \mu$ is an integral vector in the null space of $A$ with $x_i = 1$, which contradicts the integral independence of $A$. Therefore, for any row of $U_2$, its entries must have a common divisor strictly greater than $1$, and therefore a prime that is at least $2$.
\end{proof}

The following theorem is the key to derive Theorem \ref{thm:main} of this paper. It establishes a lower bound for the largest $m\times m$ subdeterminant of a matrix with integrally independent columns. 

\begin{theorem}
\label{thm:integer:equation:support}
    Let matrix $A\in \mathbb{Z}^{m\times n}$ be of full row rank with integrally independent columns. Let $p_i$ be the $i$-th prime. Then,
    \begin{equation*}
        \frac{\Gamma(A)}{\gcd(A)}\geq p_2^m p_3^m\cdots p_{\rounddown{\frac{n}{m}}}^{m}p_{\roundup{\frac{n}{m}}}^{n-m\rounddown{\frac{n}{m}}}.
    \end{equation*}
\end{theorem}

To prove this theorem, we need to introduce Jacobi's formula (see Theorem 2.5.2 of \cite{prasolov1994problems}), which uses linear algebra to establish the relation between the subdeterminant of a matrix and its inverse.



\begin{lemma}{(Jacobi's formula)}\label{lemma:jacob}
Let $A\in \mathbb{R}^{n\times n}$, $A=(a_{ij})_1^n$. For any $I, J\subseteq [n]$, denote $A_{I\times J}$ be the submatrix of $A$ consisting of rows indexed by $I$ and columns indexed by $J$. Let $A^*=(a_{ij}^*)_1^n$ be the \textit{classical adjoint} of $A$, where 
$$a^*_{ij}=(-1)^{i+j}\det(A_{[n]\setminus \{j\}~\times~ [n]\setminus \{i\}})$$
is the \textit{cofactor} of element $a_{ji}$ in $A$.
Let 
$\sigma=\begin{pmatrix} 
	i_1 & i_2 & \cdots & i_n \\ 
        j_1 & j_2 & \cdots & j_n
\end{pmatrix}$ 
be an arbitrary permutation and $(-1)^\sigma$ be its sign. Let $I=\{i_1,...i_p\}$, $J=\{j_1,...,j_p\}$, $I'=\{i_{p+1},...,i_n\}$, $J'=\{j_{p+1},...,j_n\}$ for some $1\leq p\leq n-1$. Then,
$$\det(A^*_{I\times J})=(-1)^\sigma \det(A_{J'\times I'})\cdot \det(A)^{p-1}.$$
\end{lemma}

To prove Theorem \ref{thm:integer:equation:support}, we will use the following direct corollary of Lemma \ref{lemma:jacob}.
\begin{corollary}
\label{cor:jacobi}
    Let $U\in \mathbb{R}^{n\times n}$ be a unimodular matrix. Then for any $I,J\subseteq [n]$ with $1\leq |I|=|J|\leq n-1$,
    \begin{equation}
        \det((U^{-1})_{I\times J})=\pm \det(U_{[n]\backslash J~\times~ [n]\backslash I}),
    \end{equation}
    where $U_{I\times J}$ denotes the submatrix of $U$ consisting of rows indexed by $I$ and columns indexed by $J$.
\end{corollary}

\begin{proof}[Proof of Theorem~\ref{thm:integer:equation:support}]
Suppose by applying unimodular column operations, we bring $A$ into Hermite normal form $H=\begin{pmatrix} D & 0 \end{pmatrix}$, where $D\in \mathbb{Z}^{m\times m}$ is a lower triangular matrix. We obtain the relation of $AU=\begin{pmatrix} D & 0 \end{pmatrix}$, where $U\in \mathbb{Z}^{n\times n}$ is a unimodular matrix. Let $U=\begin{pmatrix} U_1 & U_2 \end{pmatrix}$, where $U_1\in \mathbb{Z}^{n\times m}, U_2\in \mathbb{Z}^{n\times (n-m)}$. 

Recall property (P3) of $U_2$ in Lemma~\ref{lemma:propertynull} saying that for each row $i$ of $U_2$, there exists some prime $q_i \geq 2$ such that $q_i\mid (U_2)_{ij}$ for all $j \in [n-m]$.
For any prime $q$, let $I_q:=\{i\in [n]: q\mid (U_2)_{ij}, \forall j\in [n-m]\}$ be the indices of rows of $U_2$ which are divisible by $q$. We show for any prime $q$, $U_2$ must have a nonsingular $(n-m)\times (n-m)$ submatrix $(U_2)_{I\times [n-m]}$ with $I\subseteq [n]\setminus I_q$, $|I|=n-m$. Suppose not. Then every nonsingular $(n-m)\times (n-m)$ submatrix of $U_2$ must include at least one row whose index is in $I_q$. Since such a row is divisible by $q$, the determinant of such a submatrix must also be divisible by $q$. By property (P1) in Lemma \ref{lemma:propertynull}, there exists at least one such nonsingular submatrix. Therefore, the GCD of the $\binom{n}{n-m}$ such subdeterminants is at least $q \geq 2$, contradicting property (P2) in Lemma~\ref{lemma:propertynull}. Therefore, $|I_q|\leq n-|I|=m$ for any prime $q$. 

Recall $p_i$ is the $i$-th prime and $q_i \geq 2$ is a prime that divides row $i$ of $U_2$. The above argument implies that $U_2$ has a nonsingular $(n-m)\times (n-m)$ submatrix, denoted as $V$, whose row indices are in $[n]\setminus I_{p_1}$. Assume $V$ consists of the first $n-m$ rows of $U_2$ without loss of generality. Then, we have
$$q_1q_2\cdots q_{n-m}~\Big| ~|\det(V)|.$$

Notice that each prime $p_i$ appears at most $m$ times among $\{q_1,...,q_{n-m}\}$, since $|I_q| \leq m$ for any prime $q$. Also, by the way we pick $V$, $p_1$ does not appear in $\{q_1,...,q_{n-m}\}$. Together with the fact that $\det(V)\neq 0$, this allows us to bound below $|\det(V)|$ as
$$|\det(V)|\geq q_1q_2\cdots q_{n-m}\geq p_2^m p_3^m\cdots p_{\rounddown{\frac{n}{m}}}^{m}p_{\roundup{\frac{n}{m}}}^{n-m\rounddown{\frac{n}{m}}}.$$
The last term in the inequality shows a smallest possible combination of primes satisfying the conditions above, where the first $\rounddown{\frac{n-m}{m}}=\rounddown{\frac{n}{m}}-1$ primes starting from $p_2$ each appears $m$ times, while the $\roundup{\frac{n}{m}}$-th prime appears the remainder of $n-m\rounddown{\frac{n}{m}}$ times.

Next, we use the relation between the determinants of the submatrix of $A$ and $U$ to bound above $\det(V)$. Notice $A=\begin{pmatrix} D & 0 \end{pmatrix}U^{-1}$ and thus $A_{[m]\times [n-m+1:n]}=D\cdot (U^{-1})_{[m]\times [n-m+1:n]}$. We have
\begin{equation*}
\begin{aligned}
    |\det(A_{[m]\times [n-m+1:n]})|
    = & |\det(D)|\cdot |\det((U^{-1})_{[m]\times [n-m+1:n]})|\\
    = & \gcd(A) \cdot |\det((U^{-1})_{[m]\times [n-m+1:n]})| \\
    = & \gcd(A) \cdot |\det(U_{[n-m]\times [m+1:n]})|\\
    = & \gcd(A) \cdot |\det(V)|,
\end{aligned}
\end{equation*}
where the second equality follows from Fact \ref{fact:1} and the third equality follows from Corollary~\ref{cor:jacobi}. Recall $\Gamma(A)$ is the largest absolute value of an $m\times m$ subdeterminant of $A$. Combining them together we obtain
\begin{equation*}
\begin{aligned}
    \frac{\Gamma(A)}{\gcd(A)}
    \geq &\frac{|\det(A_{[m]\times [n-m+1:n]})|}{\gcd(A)}\\
    =&|\det(V)|\\
    \geq &p_2^m p_3^m\cdots p_{\rounddown{\frac{n}{m}}}^{m}p_{\roundup{\frac{n}{m}}}^{n-m\rounddown{\frac{n}{m}}},
\end{aligned}
\end{equation*}
as desired.
\end{proof}

Recall from Proposition \ref{prop:int_indep} that a matrix $A$ has integrally independent columns if and only if there exists some $b$ such that every integer solution to $Ax=b$ has full support. Taking advantage of this observation, Theorem \ref{thm:integer:equation:support} can be naturally adapted to derive an upper bound for $f(A)$, i.e. the maximum, taken over $b\in\mathcal{L}(A)$, of smallest support size of an integer solution to $Ax=b$ (formally defined in \eqref{eq:def:min:support}). This yields the proof of one of our main results.

\begin{proof}[Proof of Theorem \ref{thm:main}]
Take $b$ such that $Ax=b$ has maximum smallest support size of an integer solution to $Ax=b$, i.e. $\min~\{\|x\|_0 \mid Ax = b, x \in \Z^n\}=f(A)$. Let $x^*$ be an integral solution to $Ax=b$ with smallest support, i.e. $\zeronorm{x^*}=f(A)$. We can assume without loss of generality that $A\in \mathbb{Z}^{m\times f(A)}$ by deleting the columns $j$ of $A$ corresponding to $x^*_j=0$ for $j\in[n]$. This will not increase $\Gamma(A)$, the largest $m\times m$ subdeterminant of $A$, and will not decrease $\gcd(A)$, the GCD of all $m\times m$ subdeterminants of $A$ (this is easily seen by the fact that redundant columns can be zeroed out using unimodular column operations). Thus it suffices to prove inequality \eqref{eq:integer:equation:bound} after deleting redundant columns of $A$. By the minimality of support size of $x^*$, any integer solution to $Ax=b$ has full support. By Proposition~\ref{prop:int_indep}, columns of $A$ are integrally independent. It follows from Theorem \ref{thm:integer:equation:support} that inequality \eqref{eq:integer:equation:bound} holds. 
\end{proof}

\begin{remark}
    We demonstrate how to modify the proof of Theorem \ref{thm:integer:equation:support} to obtain the bound \eqref{eq:bound:Omega} given in \cite{aliev2022sparse}. We use the same notation as in the proof of Theorem \ref{thm:integer:equation:support}. For any $m\times m$ submatrix of $A$ whose columns are indexed by $J$ where $|J|=m$, we have
\begin{equation*}
\begin{aligned}
    |\det(A_{[m]\times J})|
    = & |\det(D)|\cdot |\det((U^{-1})_{[m]\times J})|\\
    = & \gcd(A) \cdot |\det((U^{-1})_{[m]\times J})| \\
    = & \gcd(A) \cdot |\det(U_{[n]\backslash J ~\times~ [m+1:n]})|.
\end{aligned}
\end{equation*}
    We also know that
\begin{equation*}
    \prod_{i\in [n]\backslash J}q_i~ \Big|~|\det(U_{[n]\backslash J ~\times~ [m+1:n]})|,
\end{equation*}
and thus 
\begin{equation*}
    \prod_{i\in [n]\backslash J}q_i~ \Big|~\frac{|\det(A_{[m]\times J})|}{\gcd(A)},
\end{equation*}
where $q_i, i\in [n]\backslash J$ are prime numbers with the same prime repeating at most $m$ times in $\{q_i\mid i\in  [n]\backslash J\}$.
Recall notation $\Omega_m(z)=\sum_{i=1}^{k}\min\{s_i,m\}$ for the prime factorization of $z=r_1^{s_1}\cdots r_k^{s_k}$ with multiplicities $s_1,...,s_k\in\Z_{>0}$. Clearly, when $x\mid y$, $\Omega_m(x)\leq \Omega_m(y)$. Thus, $\Omega_m\left(\prod_{i\in [n]\backslash J}q_i\right)\leq \Omega_m\left(\frac{|\det(A_{[m]\times J})|}{\gcd(A)}\right)$. Moreover, since the multiplicity of each $q_i$ in $\prod_{i\in [n]\backslash J}q_i$ is at most $m$, we have $\Omega_m\left(\prod_{i\in [n]\backslash J}q_i\right)=|[n]\backslash J|=n-m$. Therefore, $n-m\leq \Omega_m\left(\frac{|\det(A_{[m]\times J})|}{\gcd(A)}\right)$. Since $J$ is an arbitrary subset of $[n]$ with cardinality $m$, we obtain $n\leq m+\min_{\tau\in\binom{[n]}{m}, \det(A_\tau)\neq 0}\Omega_m\Big(\frac{|\det(A_\tau)|}{\gcd(A)}\Big)$. Applying the same argument as in the proof of Theorem \ref{thm:main} above, we obtain $f(A)\leq m+\min_{\tau\in\binom{[n]}{m}, \det(A_\tau)\neq 0}\Omega_m\Big(\frac{|\det(A_\tau)|}{\gcd(A)}\Big)$.
\end{remark}

\subsection{Upper bound on $h(m,t)$}
\label{sec:bound_h}


Recall $h(m,t)$ is the maximum, taken over all integer matrices $A$ with $m$ rows and largest absolute entry $t$, of the smallest support size of an integer solution to $Ax=b$ for some $b\in \mathcal{L}(A)$ (formally defined in \eqref{eq:def:maxmin:support}). Using a relation between the size of largest absolute entry of a matrix and the size of its subdeterminant, we want to use results in Section \ref{sec:bound_f} to obtain an upper bound for $h(m,t)$. We will use the well-known Hadamard inequality~\cite{hadamard1893resolution}, which gives us an upper bound on the determinant of a matrix, i.e. for any matrix $B\in \mathbb{R}^{m\times m}$, 
$$|\det(B)|\leq (\sqrt{m}\inftynorm{B})^m.$$ 
Applying this to inequality~\eqref{eq:integer:equation:bound}, we obtain the following corollary.
\begin{corollary}
\label{cor:integer:equation:support}
    $h(m,t)$ satisfies
    \begin{equation}
    \label{eq:integer:equation:bound2}
        (\sqrt{m}~t)^m\geq p_2^m p_3^m\cdots p_{\rounddown{\frac{h(m,t)}{m}}}^{m}p_{\roundup{\frac{h(m,t)}{m}}}^{h(m,t)-m\rounddown{\frac{h(m,t)}{m}}}.
    \end{equation}
\end{corollary}
\begin{proof}
Let $A$ be a matrix with $m$ rows, full row rank, and $\inftynorm{A}\leq t$ such that $\min\{\|x\|_0 \mid Ax = b, x \in \Z^{n}\}=h(m,t)$ for some $b\in \mathcal{L}(A)$. 
Then inequality \eqref{eq:integer:equation:bound2} follows directly from Theorem \ref{thm:main} and Hadamard's inequality.
\end{proof}

Furthermore, we can prove the upper bound on $h(m,t)$ of Theorem \ref{thm:integer:equation:support:bound} by applying prime number theorem (see e.g. \cite{hardy1979introduction} Theorem 6) to approximate the product of the first $k$ primes $\prod\limits_{i=1}^{k}p_i$.

\begin{proof}[Proof of Theorem \ref{thm:integer:equation:support:bound}]

According to the prime number theorem, the product of the first $k$ primes $\prod\limits_{i=1}^{k}p_i\sim e^{(1+o(1))k\log k}$, where $\log(\cdot)$ is the natural logarithm (this follows from the Prime Number Theorem, see Theorem 6 of \cite{hardy1979introduction} and Sequence A002110 of \cite{oeis}). Let $n=h(m,t)$. Relaxing~\eqref{eq:integer:equation:bound2}, we have
    $$(\sqrt{m}~t)^m\geq p_2^m p_3^m\cdots p_{\rounddown{\frac{n}{m}}}^{m}p_{\roundup{\frac{n}{m}}}^{n-m\rounddown{\frac{n}{m}}}\geq(p_1 p_2\cdots p_{\rounddown{\frac{n}{m}}})^{m}/2^m.$$
    Taking logarithm and applying $\prod\limits_{i=1}^{k}p_i\sim e^{(1+o(1))k\log k}$, we have
    $$\log(\sqrt{m}~t)\geq \log(p_1 p_2\cdots p_{\rounddown{\frac{n}{m}}})-\log 2\geq C\rounddown{\frac{n}{m}}\log \rounddown{\frac{n}{m}}$$ for some constant $C$.
    


    
    We claim that $\rounddown{\frac{n}{m}}=O\left(\frac{\log(\sqrt{m}~t)}{\log\log(\sqrt{m}~t)}\right)$. Suppose not, we assume $\rounddown{\frac{n}{m}}>\frac{2\log(\sqrt{m}~t)}{C\log\log(\sqrt{m}~t)}$ for some $m$ that is arbitrarily large. Thus,
    $$\log(\sqrt{m}~t)\geq C\rounddown{\frac{n}{m}}\log \rounddown{\frac{n}{m}}> \frac{2\log(\sqrt{m}~t)}{\log\log(\sqrt{m}~t)}\Big(\log\log(\sqrt{m}~t)-\log\big(C\log\log(\sqrt{m}~t)\big)\Big).$$
    This will give us 
    $$2\log\big(C\log\log(\sqrt{m}~t)\big)>\log\log(\sqrt{m}~t),$$
    which is not going to hold when $m$ is arbitrarily large, a contradiction.
    Thus, we have $\rounddown{\frac{n}{m}}=O\left(\frac{\log(\sqrt{m}~t)}{\log\log(\sqrt{m}~t)}\right)$. It follows that $n=O\left(\frac{m\log(\sqrt{m}~t)}{\log\log(\sqrt{m}~t)}\right)$.
\end{proof}

\subsection{Lower bound on $f(A)$}
\label{sec:tight_example}

Finally, we prove the latter statement of Theorem \ref{thm:main} by giving an example integer matrix $A$ with arbitrarily large number of rows $m$ and $\inftynorm{A}$, showing that the upper bound on $f(A)$ is asymptotically tight. A similar construction has appeared in \cite{aliev2018support,aliev2022sparse} to prove lower bounds on support size of integer solutions.

\begin{prop}
    For arbitrarily large $(m,t)$, there exists $A$ with $m$ rows, full row rank, and $\inftynorm{A}>t$ such that equality holds in~\eqref{eq:integer:equation:bound}. In other words, the bound in Theorem \ref{thm:main} is asymptotically tight.
\end{prop}
\begin{proof}
    Let $p_i$ be the $i$-th prime. There exists $k$ such that $p_2p_3\cdots p_k>t$. Let $n=km$. Define $A\in\mathbb{Z}^{m\times n}$ as
    \begin{equation*}
        \begin{aligned}
            A_{ij}=\begin{cases}
                p_1p_2\cdots p_k/p_r & ~~j=(i-1)k+r,~ 1\leq r \leq k\\
                ~~~0 & ~~\text{otherwise}
            \end{cases}
        \end{aligned}
    \end{equation*}
    and $b=A\mathbf{1}$. Then, $\inftynorm{A} = p_2p_3\cdots p_k>t$. For any $i\in[m]$, $b_i=\sum_{r=1}^{k}p_1p_2\cdots p_k/p_r$. Observe that $p_j\nmid b_i$, for any $i=1,...,m, j=1,...,k$. Any $x\in\Z^n$ with $\zeronorm{x}<n$ would have $x_j=0$ for some $j=(i-1)k+r$ with $1\leq r\leq k$. Thus, $(Ax)_i=\sum_{s=1,s\neq r}^{k}(p_1p_2\cdots p_k/p_s) x_{(i-1)k+s}$. Since $p_r\mid (p_1p_2\cdots p_k/p_s)$, $\forall s\neq r$, we have $p_r\mid (Ax)_i$. Thus, $Ax\neq b$, which means any integer $x$ with $\zeronorm{x}<n$ is not a solution to $Ax=b$. Therefore, any integer solution of $Ax=b$ has smallest support size $n=km$. Due to the block structure of $A$, the Hermite normal form of $A$ is $\begin{pmatrix} I_m & 0 \end{pmatrix}$ since $\gcd(\{A_{ij}:j=(i-1)k+r, 1\leq r\leq k\})=1, \forall i\in [m]$. Therefore, $\gcd(A)=\det(I_m)=1$. Moreover,
    $\Gamma(A)=\max\limits_{B} ~\{|\det(B)|\mid \text{$B$ is an $m\times m$ submatrix of $A$}\}=p_2^mp_3^m\cdots p_k^m$, matching the right hand side in~\eqref{eq:integer:equation:bound}.
\end{proof}


\section*{Acknowledgements}
This work was supported by ONR grant N00014-22-1-2528 and a Balas PhD Fellowship. We acknowledge the invaluable discussions with Ahmad Abdi, Gérard Cornuéjols, Levent Tunçel, and Anthony Karahalios without whom this project would not have been possible. We thank Ahmad Abdi and Timm Oertel for feedback that very much improved the presentation of these results. We also thank two anonymous reviewers for carefully reading the paper and giving useful suggestions.


\bibliographystyle{alpha} 
\bibliography{reference} 

\end{document}